%% file: main.tex
\documentclass[11pt, dvipsnames]{article}
\usepackage[margin=1in]{geometry}
\pdfoutput =1
\usepackage{setspace}
\usepackage{amsmath}
\usepackage{amssymb}
\usepackage[linesnumbered, ruled,lined]{algorithm2e}
\usepackage{varwidth}
\usepackage{booktabs}
\usepackage{algpseudocode}
\allowdisplaybreaks
\usepackage{hyperref}
\makeatletter
\newcommand{\printfnsymbol}[1]{%
  \textsuperscript{\@fnsymbol{#1}}%
}
\makeatother
\usepackage{graphicx}
\usepackage{caption}
\usepackage{color}
\usepackage{bbm}
\usepackage{float}

\usepackage{amsthm}
\usepackage{dsfont}
\usepackage{mathrsfs}
\usepackage{lmodern}
\usepackage{hyperref}

\usepackage{mathtools}
\usepackage{verbatim}
\usepackage{subcaption}
\usepackage{enumitem}
\usepackage{amsthm}
\usepackage{macros}
\newcommand{\ra}{\rightarrow}

\usepackage{booktabs}

\usepackage{xcolor}

\usepackage{appendix}
\title{Follower Agnostic Learning in Stackelberg Games
}

\author{Chinmay Maheshwari\thanks{$^{1}$CM, JC, and SS are with EECS,UC Berkeley, CA, United States.} , James Cheng$^\ast$, Shankar Sastry$^\ast$, Lillian Ratliff
\thanks{LR is with ECE, UW Seattle, WA, United States.
} ~ and Eric Mazumdar 
\thanks{EM is with CMS and Economics, Caltech, CA, United States.
} 
}

\begin{document}
\maketitle

\begin{abstract}
    In this paper, we present an efficient algorithm to solve online Stackelberg games, featuring multiple followers, in a follower-agnostic manner. Unlike previous works, our approach works even when leader has no knowledge about the followers' utility functions or strategy space.  Our algorithm introduces a unique gradient estimator, leveraging specially 
designed strategies to probe followers. In a departure from traditional assumptions of optimal play, we model followers' responses using a convergent adaptation rule, allowing for realistic and dynamic interactions. 
The leader constructs the gradient estimator solely based on observations of followers' actions. We provide both non-asymptotic convergence rates to stationary points of the leader's objective and demonstrate 
asymptotic convergence to a \emph{local Stackelberg equilibrium}. To validate the effectiveness of our algorithm, we use this algorithm to solve the problem of incentive design on a large-scale transportation network, 
showcasing its robustness even when the leader lacks access to followers' demand.
\end{abstract}
\newpage 

\section{INTRODUCTION}
\input{Section/Introduction}

\section{PROBLEM FORMULATION}\label{sec: ProblemFormulation}
\input{Section/problemFormulation}

\section{ALGORITHM AND ANALYSIS}\label{sec: Algorithm}
\input{Section/analysis}

\section{NUMERICAL EXPERIMENTS}\label{sec: Numerics}
\input{Section/numerics}

\section{CONCLUSIONS}\label{sec: Conclusions}
\input{Section/conclusions}

\bibliography{refs}
\bibliographystyle{alpha}

\appendix

\section{Technical Results}
\input{Appendix/UsefulResults}
\section{Proof of Lemmas in Main Paper}
\input{Appendix/auxiliaryLemma}
\end{document}

%% file: Section/Introduction.tex
Stackelberg games encompass a wide range of practical problems including incentive design, Bayesian persuasion, inverse optimization, bilevel optimization, cybersecurity, adversarial learning (\cite{liu2021inducing, aswani2018inverse,bard2013practical,pedregosa2016hyperparameter,liu2021investigating,sinha2017review,colson2007overview,li2022differentiable,yue2017stackelberg, maheshwari2022zeroth}), to name a few. 
Stackelberg games are comprised of two type of players -- \emph{leader} and \emph{followers}\footnote{We shall interchangeably use the the word ``leader'' with ``upper-level'' and ``follower'' with ``lower-level''.}. Mathematically, they are represented as follows: 
\begin{equation}
\begin{aligned}\label{eq: BiOptInt}
    &\min_{x\in X, y\in Y}&&\highOpt(x,y) \\ 
    &\text{s.t.} && y\in S(x) := \SOL(Y,G(x,\cdot))), 
\end{aligned}
\end{equation}
where \(X\) is the leader's strategy set, \(Y\) is the followers' (joint) strategy set,
\(\highOpt:X\times Y \ra \R \) 
is the utility of the \emph{leader}, 
\(G: X\times Y\ra \R\) 
is the \emph{game Jacobian} of \emph{followers} and
\(\SOL(Y,G(x,\cdot))\) is a \emph{variational inequality problem} that denotes the equilibrium response of followers, given the strategy of leader be \(x\in X\). 
Assuming that the set \(\br(x)\) is singleton for every \(x\in X\), aka \emph{lower-level singleton assumption}, \eqref{eq: BiOptInt} is equivalent to optimizing the following \emph{hyper-objective}:
\begin{align}\label{eq: HyperObj}
    \min_{x\in X} \tilde{f}(x) := f(x,\br(x)).
\end{align}
Note that in general  \eqref{eq: HyperObj} is non-convex optimization problem. Thus, the goal in Stackelberg games is to find a stationary point / local optima of \eqref{eq: HyperObj} (\cite{fiez2019convergence}). 

In numerous practical scenarios, it is unrealistic to presume that the leader possesses any information regarding the variational inequality problem at the lower-level, including 
\(\SOL(\cdot)\) and even their strategy set \(Y\) -- information traditionally assumed in prior research on solving Stackelberg games. 
Thus, the key question we ask in this work is: 
\begin{quote}
    \emph{Can we design efficient algorithms for Stackelberg games where the leader does not require explicit knowledge of the game played between followers?}
\end{quote}
In this work, we affirmatively answer the above question in the setting where the leader can only probe the followers with different strategies and receive estimates of their (approximated) equilibrium responses. This is in contrast to the common assumption in the literature on Stackelberg games, where it is assumed that the leader has access to an equilibrium or best-response of followers either by knowledge of the utility function of followers or through an oracle. In particular, we consider that followers are rational in the sense that they employ an adaptation/learning algorithm, which asymptotically converges to the equilibrium (\cite{fudenberg1998theory}).

We propose a \emph{two-loop} algorithm where, in the outer loop, the leader fixes its strategy (i.e., the value of $x$) and announces it to the followers. Between two updates of the leader's strategy, the followers employ an adaptation algorithm, for a finite number of steps, so that they converge to an \emph{approximate} equilibrium (or best-response). Upon observing the followers' behavior, the leader constructs an approximate estimator of the gradient of the hyper-objective \eqref{eq: HyperObj} and updates its strategy via gradient descent using the estimator.

We show that the proposed algorithm converges to the stationary point of \eqref{eq: HyperObj} at a rate \(\mathcal{O}(T^{-1/2})\). Moreover, we show that if the hyper-objective satisfies the \emph{strict-saddle property}, i.e. the minimum eigenvalue at any saddle point is strictly negative, then the iterates asymptotically avoid saddle points (which include local maxima) and converge to a local minima of the hyper-objective (aka local Stackelberg equilibrium \cite{fiez2019convergence}). 

We corroborate the theoretical results by conducting a simulated study of the proposed algorithm to design tolls over the Sioux Falls (South Dakota, US) transportation network. In this setup, we assume that the leader does not know the origin-destination (o-d) demand of travelers moving between different o-d pairs, which is sensitive information.

\subsection{Related works}
\textbf{Learning in Stackelberg games:} Learning in Stackelberg games with \emph{finite} actions is a very active area of research (\cite{blum2014learning, letchford2009learning, peng2019learning, bai2021sample, sessa2020learning}), where the leader has access to either a noisy or exact best response oracle. Furthermore, a dominant paradigm in this literature is to consider two-player games with finite strategy sets or linearly parametrized utility functions, with the exception of \cite{fiez2019convergence} where the author studies the convergence of a two-timescale algorithm to the Stackelberg equilibrium, requiring knowledge of the Hessian of followers' utility functions for leader updates. However, in this work, we aim to design follower-agnostic learning in a general-sum Stackelberg game in continuous spaces with \emph{no} knowledge of the followers' utility functions.

\textbf{Bilevel optimzation.} 
Bilevel optimization, a subset of problem \eqref{eq: BiOptInt}, is extensively studied in literature, resembling a Stackelberg game with a single leader and follower. Existing research on bilevel optimization pursues three main approaches. The first utilizes a value function-based approach, converting the problem into a constrained single-level optimization problem with convergence guarantees to approximate Karush-Kuhn-Tucker (KKT) points \cite{sow2022constrained, ye2022bome}. However, such points may not capture locally optimal solutions \cite{chen2023bilevel}. Another line of research focuses on asymptotic convergence of solutions of simpler bilevel problems than \eqref{eq: BiOptInt} under various assumptions on the lower-level objective function structure \cite{liu2020generic,liu2021value,liu2021towards}. The third line explores solving the non-convex optimization problem \eqref{eq: HyperObj} using gradient descent, requiring the computation of the gradient of the solution mapping, denoted as \(\nabla \br(x)\). While many methods exist for approximating \(\nabla \br(x)\), including \emph{Automatic Implicit Differentiation} (AID) (\cite{grazzi2020iteration,franceschi2017forward,pedregosa2016hyperparameter,ji2020bilevel,franceschi2018bilevel}), or  \emph{Iterative Differentiation}  (\cite{franceschi2017forward,ghadimi2018approximation,gould2016differentiating,shaban2019truncated}), our work is closely related to zeroth-order methods, specifically avoiding the computation of the Hessian (\cite{chen2023bilevel}). Our proposed algorithm shares similarities with \cite{chen2023bilevel}, but we eliminate the need for oracle access to a lower-level optimal solution, leveraging two-timescale stochastic approximation to analyze accumulated errors \cite{chen2023bilevel}.

%% file: Section/problemFormulation.tex
Consider the following Stackelberg game
\begin{align}\label{eq: SG}
    &\min_{x\in X, y\in Y} && \highOpt(x,y) \notag \\ & \text{such that}  && y\in S(x) := \SOL(Y,G(x,\cdot))),\tag{\textsf{SG}}
\end{align}
where \textit{(i)} \(X = \R^d\) and \(Y\subset \R^{d'}\) is assumed to be convex and compact set; \textit{(ii)} \(\highOpt:X\times Y \ra \R \) and \(G: X\times Y\ra \R\) are twice continuously differentiable functions; \textit{(iii)} \(\SOL(Y,G(x,\cdot))\)  denotes the solution to variational inequality characterized by functional \(G(x,\cdot)\). That is, \( 
    \SOL(Y,G(x,\cdot)) = \{y\in Y: \langle y' - y, G(x,y) \rangle \geq 0, \quad \forall \ y'\in Y\}.
\)
Under mild conditions on the monotonicity of \(G(x,\cdot)\), it is ensured that \(S(x)\) is non-empty and convex (\cite{facchinei2003finite}). 
 
 In what follows, we call a continuously differentiable function \(\tilde{f}:\R^d\ra \R\) to be \(L-\)Lipschitz if for every \(x,x'\in \mathbb{R}^d\), 
 \begin{align*}
 \|\tilde{f}(x)-\tilde{f}(x')\| \leq L\|x-x'\|.    
 \end{align*}
  Furthermore, we call it to be \(\ell-\)smooth if for every \(x,x'\in \mathbb{R}^d,\)
  \begin{align*}
\|\nabla \tilde{f}(x)- \nabla \tilde{f}(x')\| \leq \ell \|x-x'\|.      
  \end{align*}

Next, we introduce the main assumptions on the parameters of \eqref{eq: SG} made throughtout this paper 
\begin{assumption}\label{assm: BasicAssumptionSetup}
    \begin{itemize}
        \item[(1)] For every \(y\in Y\), the function \(f(\cdot,y)\) is \(\fxLipschitz\)-Lipschitz.
Additionally, for every \(x\in X\), the function \(f(x,\cdot)\) is
\(\fyLipschitz-\)Lipschitz and \(\fySmooth\)-smooth.
\item[(2)] For every \(x\in X,\) the set \(S(x)\) is singleton and function \(S(x)\) is \(\brLipschitz\)-Lipschitz.
\item[(3)] The function \(\tilde{f}(x) = f(x,S(x))\) is {twice-continuously differentiable}, \(\tildefLipschitz\)-Lipschitz and \(\tildefSmooth-\)smooth. 
    \end{itemize}
\end{assumption}

\section{MOTIVATING EXAMPLE: INCENTIVE DESIGN IN ROUTING GAMES}\label{sec: MotivatingExample}
Consider a transportation network \(\mc{G} =(\mc{N},\mc{E})\) comprised of set of nodes \(\mc{N}\) and set of edges \(\mc{E}\), used by self-interested (infinitesimal)travelers. Each traveler is traveling between some origin-destination (o-d) pair on the network. The set of all o-d pairs be denoted by \(\mc{Z}\). For each o-d pair \(z \in \mc{Z}\), let \(\mc{R}_z\) be the set of routes connecting the o-d pair \(z\). Let \(D_z\) be the demand of travelers traveling between o-d pair \(z\in \mc{Z}\) and 
\(y_{rz}\) be the flow of 
travelers from o-d pair \(z\in \mc{Z}\) that choose route \(r\in \mc{R}_z\). Naturally, \(\sum_{r\in\mc{R}_z}y_{rz}=D_z,\) for every \(z\in \mathcal{Z}\). We denote the set of all feasible route flows by \(Y = \prod_{z\in \mc{Z}}Y_z,\) where \(Y_z := D_z\cdot \Delta(\mb{R}^{|\mc{R}_z|})\) is a simplex.
The route flow gives rise of 
congestion on the edges of the network. Given a route flow \(y\in Y\), the resulting congestion on edges is denoted by \(w(y) = (w_e(y))_{e\in \mc{E}},\) where 
\begin{align}\label{eq: EdgeCongestion}
w_e(y) = \sum_{z\in \mc{Z}}\sum_{r\in \mc{R}_z} y_{rz}\mathbf{1}(e\in r), \quad \forall \ e\in \mathcal{E}.    
\end{align}
Higher congestion leads to higher travel time on any edge. More formally, let \(\ell_e(\cdot)\) be a strictly increasing smooth function which denotes the travel time of using edge \(e\in \mc{E}\) as a function of congestion.
A social planner can alter the congestion levels on the network by imposing tolls on the edges of the network which changes the preferences of travelers for different routes. Let \(x_e\in \mb{R}\) denote the tolls imposed on edge  \(e\in \mc{E}\). \footnote{Here, we allow for tolls to take negative values. Such tolling scheme can  be implemented by considering revenue-refunding schemes.} Under the network tolls \(x=(x_e)_{e\in \mc{E}}\in \mathbb{R}^{|\mc{E}|}\) and route flow \(y\in Y\), the overall cost experienced by travelers from o-d pair \(z\in \mc{Z}\) taking a route \(r\in \mc{R}_z\) is
\begin{align}\label{eq: route_cost}
    c_{r}(y,x) = \sum_{e\in r}\ell_e(w_e(y)) + x_e.
\end{align}

Given a fixed network tolls \(x\), the resulting congestion -- \emph{Wardrop equilibrium} -- can be obtained by solving the following strictly convex optimization problem (\cite{patriksson2015traffic})
{\begin{align}\label{eq: WE}
    S(x) = \underset{y\in Y}{\arg\min}~\Phi(y,x) = \sum_{e\in \mc{E}}\int_{0}^{w_e(y)}\ell_e(\theta) d\theta + x_ew_e(y). 
\end{align}}
The goal of social planner is to minimize the overall congestion on the network while also minimizing the tolls levied on travelers. More formally, the planner's objective function is given by \(
    f(x,y)= \sum_{e\in \mc{E}}w_e(y)\ell_e(w_e(y)) + \lambda \|x\|^2,
\)
where the first term corresponds to the average congestion on the network and second term is a regularization term with parameter \(\lambda> 0\), which ensures low values of tolls. 
Thus, the problem of toll design is as follows
\begin{align}\label{eq: TollingScheme}
    &\min_{x\in \mathbb{R}^{|\mc{E}|}, y\in Y} && f(x,y)\notag\\ 
    &\quad \quad \text{s.t. } && y\in S(x) = \underset{y'\in Y}{\arg\min}~\Phi(y,x),
\end{align}
which is an instantiation of \eqref{eq: SG}.
\begin{remark}
In order to compute \(S(x)\) in \eqref{eq: WE} the planner needs to know  the set \(Y\) that requires knowledge of the demand of travelers between various o-d pairs, which is a sensitive information. In Section \ref{sec: Numerics}, we use the approach in this paper to solve \eqref{eq: TollingScheme} where the incentive designer does not know the o-d demand of travelers and can only observe the congestion levels \((w_e)_{e\in\mc{E}}\) on the network in response to the set tolls.
\end{remark}

%% file: Section/analysis.tex
In this section, we present a follower agnostic algorithm for solving \eqref{eq: SG}. Following which, we present the convergence guarantees of the proposed algorithm to a stationary point. Additionally, we show that the algorithm will eventually converge to a local optima by avoiding the saddle points and local maximum.  
\subsection{Algorithmic structure}
 The algorithm is based on alternatively moving towards solution to the variational inequality at lower level and descending along the upper-level objective function.
    Specifically, between two updates of leader (upper-level), the followers (lower level) employ an iterative adaptive rule, aimed to solve the variational inequality \(\SOL(\cdot)\), for a fixed number of steps. Following which, the upper level iterates descend along an ``approximated'' gradient estimator, inspired from zeroth-order optimization (\cite{spall1997one, flaxman2004online}), evaluated at the lower-level iteration in current round.
\paragraph{Leader's strategy update}
The leader's update rule is as follows: 
\begin{align}\label{eq: UpperLevel}
    x_{t+1} = x_t - \stepX_t \ZO(x_t;\delta_t,v_t),   \tag{\textsf{UL}}
\end{align}
where \(\ZO(x;\delta,v)\) denotes a gradient estimator of function \(\tilde{f}(\cdot) := f(\cdot,\br(\cdot))\), evaluated at \(x\) with parameters \(\delta,v\). We shall describe the estimator in detail below.  
\paragraph{Gradient estimator}
In order to compute the gradient of \(\tilde{f}(x)\), we need to compute the derivative through the solution to the variational inequality in \eqref{eq: SG}, i.e. \(S(x)\), which may involve higher order gradient computations and at times is not computable in closed form due to constraints. In this work, we consider a gradient estimator inspired from  \cite{spall1997one,flaxman2004online}.  Specifically, we consider the following estimator 
\begin{align}\label{eq: GradEstimator}
    \ZO(x;\delta,v):= \frac{d}{\delta}\left( f(\hat{x},y^{(K)}(\hat{x})) - f(x,y^{(K)}(x)) \right) v,
\end{align}
such that \textit{(i)} \(\hat{x} = x+\delta v\), where \(v\in \mathcal{S}(\R^d):= \{z\in \R^d:\|z\|_2=1\}\) and \(\delta > 0\), are  referred as \emph{perturbation} and \emph{perturbation radius} respectively; \textit{(ii)} \(K\) is a positive integer capturing the number of rounds of adaptation rule employed by followers between two updates of leader's strategy; \textit{(iii)} for any \(x\in X\), \(k\in[K-1]\) consider a iterative solver for variational inequality denoted by \(H\) such that 
    \begin{align}\label{eq: LowerLevel}
        y^{(k+1)}(x) = H_k(y^{(k)}(x);x), \quad\forall \ k \in [K-1],\tag{\textsf{LL}}
    \end{align}
    where \(y^{(0)}\) is some initialization for the iterative solver of variational 
    inequality. 
    For example, when the lower level problem is just a convex optimization problem with objective function \(g(x,\cdot)\), a typical choice of \(H_k\) is projected 
 gradient descent, i.e. \( H_k(y;x)= \proj_Y(y-\stepY_k \nabla_y g(x,y)),\) where \(\proj_Y\) denotes the orthogonal projection on \(Y\) and \(\gamma_k\) is the step size. Note that, in order to construct the gradient estimator in \eqref{eq: GradEstimator}, the leader \emph{need not} know the exact description of update rule \(H_k\). For most of the paper, we shall concisely denote \(y^{(k)}(x)\) and \(y^{(k)}(\hat{x})\) as \(\tilde{y}^{(k)}\) and \(y^{(k)}\) respectively. 

\begin{remark}\label{rem: Diff_Standard_ZO}
Direct application of zeroth-order gradient estimator from \cite{spall1997one,flaxman2004online} would result in following estimator
\begin{align}\label{eq: ZOOriginal}
\ZOOrig(x;\delta,v) = \frac{d}{\delta}\left( \tilde{f}(\hat{x}) - \tilde{f}(x) \right) v,
\end{align}
where \(\tilde{f}\) is defined in \eqref{eq: HyperObj}.
Observe that the gradient estimators  \(\ZO\) and \(\ZOOrig\) differ because in \eqref{eq: GradEstimator} we evaluate \(f(x,\cdot)\) at   \(y^{(K)}(x)\) while in \eqref{eq: ZOOriginal} we evaluated it at \( \br(x)\) for any \(x\in \R^d\). This induces additional bias in the gradient estimator that needs to be appropriately accounted while establishing convergence results.  
\end{remark}

\paragraph{Algorithm}
The algorithms runs for \(T\) rounds. In every round \(t\in [T-1]\) the leader queries the followers with two strategies \(x_t\) and \(\hat{x}_t = x_t + \delta_t v_t\) where \(v_t \sim \textsf{Unif}(\mc{S}(\mb{R}^d))\) is a vector sampled uniformly randomly from the unit sphere in \(\mb{R}^d\) and \(\delta_t\) is the \emph{perturbation radius} (refer line \textbf{2}-\textbf{3} in Algorithm \ref{alg: ZerothOrderTwoPointAlgorithm}). 
\input{Section/algorithm}

The followers respond to the leader's strategies by using an iterative variational inequality solver for \(\yIter\) steps to obtain \(\tilde{y}_t^{(\yIter)}\)  and \(y_t^{(\yIter)}\) respectively (refer line \textbf{4} and \textbf{7} in Algorithm \ref{alg: ZerothOrderTwoPointAlgorithm}). 
After observing \(\tilde{y}_t^{(\yIter)}\) and \({y}_t^{(\yIter)}\), the leader computes a gradient estimator as per \eqref{eq: GradEstimator}. 
The leader updates its strategy for next time as per \eqref{eq: UpperLevel} (refer line \textbf{8} in Algorithm \ref{alg: ZerothOrderTwoPointAlgorithm}). The followers initialize their strategies as per line \textbf{9} in Algorithm \ref{alg: ZerothOrderTwoPointAlgorithm}. 

\subsection{Convergence to stationary points}

We now study the convergence properties of Algorithm \ref{alg: ZerothOrderTwoPointAlgorithm}. 
\begin{assumption}\label{assm: FollowerUpdatesDiff}
For any \(x,\hat{x}\in X\), the updates in \eqref{eq: LowerLevel} are such that \(
        \|y^{(\yIter)}(x)-y^{(\yIter)}(\hat{x})\| \leq C \|x-\hat{x}\|, 
    \)
    for some \(C>0\). 
\end{assumption}
Assumption \ref{assm: FollowerUpdatesDiff} posits that the adaptation rule employed by followers is stable with respect to perturbations in the leader's strategy. This assumption is typically satisfied by many algorithms, including gradient-based algorithms. 
\begin{assumption}\label{assm: FollowerUpdatesConvergence}
Atleast one of the following holds:
    \begin{itemize}
    \item[(1a)] For any \(x\in X,\) the iterates \eqref{eq: LowerLevel} converge to equilibrium at a polynomial rate. That is, for any initial point \(y^{(0)}\in Y\),\(
    \|y^{(K)}(x)-\br(x)\|^2 \leq CK^{-\lambda} \|y^{(0)}-\br(x)\|^2,
\)
where \(\lambda,C\) are positive scalars. 
\item[(1b)] For any \(x\in X,\) the iterates \eqref{eq: LowerLevel} converge to equilibrium at a exponential rate. That is, for any initial point \(y^{(0)}\),
\(
    \|y^{(K)}(x)-\br(x)\| \leq C\rho^K \|y^{(0)}-\br(x)\|,
\)
where \(C\) is a positive scalar and \(\rho\in [0,1)\).
    \end{itemize}
\end{assumption}

\begin{remark}
Convergence of lower-level problem is extensively studied in  literature, e.g. \cite{nesterov2018lectures,wright2022optimization},  and is not the focus of this article. Assumption \ref{assm: FollowerUpdatesConvergence}(1a) holds for gradient descent updates for convex functions that satisfy \emph{quadratic growth} condition \cite{karimi2016linear}.
Meanwhile, Assumption \ref{assm: FollowerUpdatesConvergence}(1b) holds for gradient descent on strongly convex functions.  
\end{remark}

\begin{theorem}\label{thm: ConvergenceStationary}
Let Assumption \ref{assm: BasicAssumptionSetup}-\ref{assm: FollowerUpdatesConvergence} hold. If we choose \(\eta_t = \bar{\eta} (t+1)^{-1/2}d^{-1}, \delta_t = \bar{\delta} (t+1)^{-1/4}d^{-1/2}\) such that \(\bar{\eta}\leq d/2\tildefSmooth\). 
Then the updates \((x_t)_{t\in[T]}\) in Algorithm \ref{alg: ZerothOrderTwoPointAlgorithm} are such that 
\begin{align*}
    \min_{t\in[T]}\avg\ls{\|\nabla \tilde{f}(x_t)\|^2} \leq \tilde{O}\lr{\frac{d}{\sqrt{T}} + \frac{\alpha}{1-\alpha}{d^3}\left(1+\frac{1}{\sqrt{T}}\right)},
\end{align*}
where \(\alpha = CK^{-\lambda}\) if Assumption \ref{assm: FollowerUpdatesConvergence}(1a) hold, or \(\alpha = \rho^K\) if Assumption \ref{assm: FollowerUpdatesConvergence}(1b)  hold. 
\end{theorem}
Intuitively, the theorem states that if we want to converge closer to a stationary point then we need to run the Algorithm \ref{alg: ZerothOrderTwoPointAlgorithm} with larger \(T\) or smaller \(\alpha\) (i.e. larger \(K\)). 
Crucially, the term \(\alpha d^3\) in the bound is due to error accumulation between time steps due to non-convergence of lower-level to exact solution of variational inequality \(S(x)\). Owing to such precise characterization of error accumulation across time steps, our rate
is informative of the \emph{computational complexity} of solving the bi-level problem while in other contemporary work, namely \cite{chen2023bilevel}, it resembles \emph{iteration complexity} of the oracle.
Since \(\alpha\) is a function of \(K\), the number of lower level iterations in every round, we can choose \(K\) to be large enough to make sure that the algorithm converges closer to the stationary point. 

\begin{corollary}\label{cor: Conv}
Let Assumption \ref{assm: BasicAssumptionSetup}-\ref{assm: FollowerUpdatesDiff} and Assumption \ref{assm: FollowerUpdatesConvergence}(1a) hold. Set \(\eta_t = \bar{\eta} (t+1)^{-1/2}d^{-1}, \delta_t = \bar{\delta} (t+1)^{-1/4}d^{-1/2}\) such that \(\bar{\eta}\leq d/2\tildefSmooth\). Additionally, set 
\(K \geq T^{1/{2\lambda}}d^{2/\lambda}\). 
Then the iterates of Algorithm \ref{alg: ZerothOrderTwoPointAlgorithm} satisfy 
\begin{align*}
\min_{t\in[T]}\avg\ls{\|\nabla \tilde{f}(x_t)\|^2} \leq \tilde{O}\lr{\frac{d}{\sqrt{T}}}.
\end{align*}

\end{corollary}
\begin{corollary}\label{cor: ConvExp}
Let Assumption \ref{assm: BasicAssumptionSetup}-\ref{assm: FollowerUpdatesDiff} and Assumption \ref{assm: FollowerUpdatesConvergence}(1b) hold.
Set \(\eta_t = \bar{\eta} (t+1)^{-1/2}d^{-1}, \delta_t = \bar{\delta} (t+1)^{-1/4}d^{-1/2}\) such that \(\bar{\eta}\leq d/2\tildefSmooth\). Additionally, set 
 \(K\geq (1/|\log(\rho)|)\lr{(1/2) \log(T)
 + 2\log(d)} \).
Then 
\begin{align*}
\min_{t\in[T]}\avg\ls{\|\nabla \tilde{f}(x_t)\|^2} \leq \tilde{O}\lr{\frac{d}{\sqrt{T}}}. 
\end{align*}
\end{corollary}

\begin{remark}
    We know that for non-convex smooth functions, gradient descent converges to a stationary point (at a rate of \(\mc{O}(1/\sqrt{T})\)). However, the key point of departure of \eqref{eq: UpperLevel} from standard gradient descent is the presence of bias in the gradient estimator. Consequently,
the key component of the proof is to bound the error in the gradient estimator \eqref{eq: GradEstimator}. This is because the estimator can be decomposed as \(
    \ZO(x_t;\delta_t,v_t) = \nabla \tilde{f}(x_t) + \mc{E}_t^{(1)} + \mc{E}_t^{(2)} + \mc{E}_t^{(3)},
\)
where 
\begin{align*}
    \mc{E}^{(1)}_t&:=\avg\ls{\ZOOrig(x_t;\delta_t,v_t)|x_t} -\nabla \tilde{f}(x_t), \\ 
\mc{E}^{(2)}_t&:=\ZOOrig(x_t;\delta_t,v_t)
    -\avg\ls{\ZOOrig(x_t;\delta_t,v_t)|x_t}, \\ 
\mc{E}^{(3)}_t&:= \ZO(x_t;\delta_t,v_t) - \ZOOrig(x_t;\delta_t,v_t).
\end{align*}
The term \(\mc{E}_t^{(1)}\) denotes the bias introduced due to the difference between standard zeroth-order gradient estimator, as per \eqref{eq: ZOOriginal}, and the true gradient. The term \(\mc{E}_t^{(2)}\) denotes the randomness introduced if we were to use the standard zeroth-order gradient estimator \eqref{eq: ZOOriginal}. Finally, the term \(\mc{E}_t^{(3)}\) denotes the bias introduced due to difference between our gradient estimator \eqref{eq: GradEstimator} and the standard zeroth-order gradient estimator (cf. Remark \ref{rem: Diff_Standard_ZO}).
\end{remark}

\paragraph{Proof of Theorem \ref{thm: ConvergenceStationary}}
The proof of Theorem \ref{thm: ConvergenceStationary} follows by noting that \(\tilde{f}\) approximately decreases along the trajectory \eqref{eq: UpperLevel} (Lemma \ref{lem: TildeFPotential}). Note that the decrease is said to be ``approximate'' because of the bias introduced by \eqref{eq: GradEstimator} in comparison to actual gradient \(\nabla \tilde{f}(\cdot)\). We then proceed to individually bound the bias terms (Lemma \ref{lem: ErrorBounds}). The convergence rate follows by using the step size and perturbation radius stated in the statement of Theorem \ref{thm: ConvergenceStationary}.

\textbf{Proof of Theorem \ref{thm: ConvergenceStationary}.}
From Lemma \ref{lem: TildeFPotentialAppendix} we know that \(\tilde{f}(\cdot)\) approximately decreases along the trajectory of \eqref{eq: UpperLevel}. That is, 
\begin{align}\label{eq: NowBoundError}
    &\avg\ls{\tilde{f}(x_{t+1})} \leq \avg\ls{ \tilde{f}(x_t)}-\frac{\eta_t}{2}\avg\ls{\|\nabla \tilde{f}(x_t)\|^2 } \notag \\&+ \eta_t \avg\ls{\|\mc{E}_t^{(1)}\|^2}+ \eta_t \avg\ls{\|\mc{E}_t^{(3)}\|^2}+{\tildefSmooth \eta_t^2} \avg\ls{\|\mc{E}_t^{(2)}\|^2}.
\end{align}
Using the bounds on error terms from Lemma \ref{lem: ErrorBoundsAppendix}, we obtain 
{\small \begin{align*}
     &\avg\ls{\tilde{f}(x_{t+1})} \leq \avg\ls{ \tilde{f}(x_t)}-\frac{\eta_t}{2}\avg\ls{\|\nabla \tilde{f}(x_t)\|^2 } + \eta_t \frac{\tildefSmooth^2\delta_t^2d^2}{4}\\ &+\eta_t \lr{\frac{d^2}{\delta_t^2}\fyLipschitz^2\left(2\alpha^t e_0 + 
    2Cd^2\sum_{k=0}^{t-1}\alpha^{t-k} \eta_k^2 + Cd\sum_{k=0}^{t-1}\alpha^{t-k}\delta_k^2\right)} \\&+4d^2 \tildefLipschitz^2{\tildefSmooth \eta_t^2} 
 \end{align*}}
Re-arranging the terms and adding and substracting the term \(\tilde{f}(x^\ast)=\min_{x\in X}\tilde{f}(x)\), we obtain 
 \begin{align*}
     &\frac{\eta_t}{2}\avg\ls{\|\nabla \tilde{f}(x_t)\|^2} \leq \avg\ls{\tilde{f}(x_t)}-{\tilde{f}(x^\ast)} + \eta_t\frac{\tildefSmooth^2\delta_t^2d^2}{4}\\ &\quad+  \eta_t{\frac{d^2}{\delta_t^2}\fyLipschitz^2\left(2\alpha^t e_0 + 
    2Cd^2\sum_{k=0}^{t-1}\alpha^{t-k} \eta_k^2 + C\sum_{k=0}^{t-1}\alpha^{t-k}\delta_k^2\right)} \\&\quad - \lr{ \avg\ls{\tilde{f}(x_{t+1})} - {\tilde{f}(x^\ast)}  } +4d^2 \tildefLipschitz^2{\tildefSmooth \eta_t^2}. 
 \end{align*}
 Summing the previous equation over time step \(t\) we obtain 
 \begin{align}\label{eq: SummingOver}
   & \sum_{t\in[T]}\eta_t\avg\ls{\|\nabla \tilde{f}(x_t)\|^2} \leq \lr{\tilde{f}(x_0)-\tilde{f}(x^\ast)} \notag \\&\quad  + \frac{\tildefSmooth^2d^2}{4}\sum_{t\in[T]}\eta_t\delta_t^2 + 2e_0d^2\fyLipschitz^2 \sum_{t\in[T]}\frac{\eta_t}{\delta_t^2 }\alpha^t\notag \\&\quad + \underbrace{2Cd^4\fyLipschitz^2  \sum_{t\in[T]}\frac{\eta_t}{\delta_t^2}\sum_{k=0}^{t-1}\alpha^{t-k}\eta_k^2}_{\text{Term E}} \notag \\&\quad + \underbrace{C\fyLipschitz^2d^2\sum_{t\in[T]}\frac{\eta_t}{\delta_t^2}\sum_{k=0}^{t-1}\alpha^{t-k}\delta_k^2}_{\text{Term F}}+ 4d^2 \tildefLipschitz^2{\tildefSmooth}\sum_{t\in[T]}\eta_t^2
 \end{align}

Setting \(\eta_t = \bar{\eta} (t+1)^{-1/2}d^{-1}\) and \(\delta_t =\bar{\delta} (t+1)^{-1/4}d^{-1/2}\), as per the statement of Theorem \ref{thm: ConvergenceStationary}, and dividing both sides by \(\sum_{t\in [T]}\eta_t\), we obtain 
\begin{align*}
    &\frac{1}{\sum_{t\in[T]}\eta_t}\sum_{t\in[T]}\eta_t\avg\ls{\|\nabla \tilde{f}(x_t)\|^2} \leq \frac{Cd}{\bar{\eta}\sqrt{T}}\lr{\tilde{f}(x_0)-\tilde{f}(x^\ast)} \\&\quad + \frac{C\tildefSmooth d\log(T)\bar{\delta}^2}{4\sqrt{T}} + \frac{2Cd^3\fyLipschitz^2\alpha}{(1-\alpha )\bar{\eta}\sqrt{T}}  \notag \\&\quad +\frac{1}{\sum_{t\in[T]}\eta_t}\underbrace{2d^4C\fyLipschitz^2  \sum_{t\in[T]}\frac{\eta_t}{\delta_t^2}\sum_{k=0}^{t-1}\alpha^{t-k}\eta_k^2}_{\text{Term E}} \\&\quad+ \frac{1}{\sum_{t\in[T]}\eta_t}\underbrace{\fyLipschitz^2Cd^2\sum_{t\in[T]}\frac{\eta_t}{\delta_t^2}\sum_{k=0}^{t-1}\alpha^{t-k}\delta_k^2}_{\text{Term F}}\\&\quad+ \frac{4C\tildefLipschitz^2\tildefSmooth^2\bar{\eta}\log(T) d}{\sqrt{T}},
\end{align*}
where \(C\) is a positive scalar. Next, we bound \(\text{Term E}+\text{Term F}\) as follows 
\begin{align}\label{eq: EPF}
    &\text{Term E } + \text{Term F} \leq  \sum_{t=1}^{T}\frac{\eta_t}{\delta_t^2} \sum_{k=0}^{t-1}\alpha^{t-k}\left(d^4\eta_k^2+d^2\delta_k^2\right) \notag \\ 
    &=\frac{\bar{\eta}}{\bar{\delta}^2}  \sum_{t=1}^{T}\sum_{k=0}^{t-1} \alpha^t \frac{\Theta_k}{\alpha^k} = \frac{\bar{\eta}}{\bar{\delta}^2} \sum_{k=0}^{T-1}\frac{\Theta_k}{\alpha^k} \sum_{t=k+1}^{T}  \alpha^t \notag  \\ 
    &\leq\frac{\bar{\eta}}{\bar{\delta}^2}  \sum_{k=0}^{T-1}\frac{\Theta_k}{\alpha^k}\frac{\alpha^{k+1}}{1-\alpha}  =\frac{\bar{\eta}}{\bar{\delta}^2}  \frac{\alpha}{1-\alpha} \sum_{k=0}^{T-1}{\Theta_k} \notag \\
    &= \frac{\bar{\eta}}{\bar{\delta}^2}  \frac{C\alpha}{1-\alpha}\lr{d^2\bar{\eta}^2\log(T)+d^2\bar{\delta}^2\sqrt{T}},
\end{align}
where in second equality \(\Theta_k := (d^4\eta_k^2+d^2\delta_k^2)\), and we have appropriately adjusted the constant \(C\) to account for additinal constants.  
Thus, combining \eqref{eq: SummingOver} and \eqref{eq: EPF}, we obtain

\begin{align*}
    &\frac{1}{\sum_{t\in[T]}\eta_t}\sum_{t\in[T]}\eta_t\avg\ls{\|\nabla \tilde{f}(x_t)\|^2} \\&\leq \mathcal{O}\bigg(\frac{d}{\sqrt{T}}\lr{\tilde{f}(x_0)-\tilde{f}(x^\ast)} + \frac{ d\log(T)}{\sqrt{T}}  \notag \\&\quad +  \frac{d^3\alpha}{(1-\alpha )\sqrt{T}} + \frac{4\log(T) d}{\sqrt{T}} \\ &\quad +\frac{d}{\sqrt{T}}  \frac{\alpha}{1-\alpha}\lr{d^2\log(T)+d^2\sqrt{T}}\bigg). 
\end{align*}

To conclude, we obtain 
\begin{align*}
    &\min_{t\in[T]}\avg\ls{\|\nabla \tilde{f}(x_t)\|^2} \leq \frac{1}{\sum_{t\in[T]}\eta_t}\sum_{t\in[T]}\eta_t\avg\ls{\|\nabla \tilde{f}(x_t)\|^2} \\&\leq \tilde{O}\lr{\frac{d}{\sqrt{T}} + \frac{\alpha}{1-\alpha}{d^3}\left(1+\frac{1}{\sqrt{T}}\right)}. 
\end{align*}
This concludes the proof. \qed

Now, we formally state the Lemmas used in the proof. 
\begin{lemma}\label{lem: TildeFPotential}
If \(\bar{\eta}\leq d/(2\tildefSmooth)\) then 
\begin{equation}
\begin{aligned}\label{eq: NowBoundErrorMP}
    &\avg\ls{\tilde{f}(x_{t+1})} \leq \avg\ls{ \tilde{f}(x_t)}-\frac{\eta_t}{2}\avg\ls{\|\nabla \tilde{f}(x_t)\|^2 } \\ & + \eta_t \avg\ls{\|\mc{E}_t^{(1)}\|^2}+ \eta_t \avg\ls{\|\mc{E}_t^{(3)}\|^2}+{\tildefSmooth \eta_t^2} \avg\ls{\|\mc{E}_t^{(2)}\|^2}.
\end{aligned}
\end{equation}
\end{lemma}
The proof of Lemma \ref{lem: TildeFPotential} follows in two steps: First, we use second-order Taylor series expansion of \(\tilde{f}\) along the iterate values. Second, we use \eqref{eq: UpperLevel} and complete the squares using algebraic manipulations. A detailed proof is provided in the Appendix. 
\begin{lemma}\label{lem: ErrorBounds}
The errors \(\avg\ls{\|\mc{E}_t^{(i)}\|^2 }\) for \(i\in \{1,2,3\}\) are bounded as follows:
\begin{align*}
    \avg\ls{\|\mc{E}_t^{(1)}\|^2}&\leq \frac{\tildefSmooth^2\delta_t^2d^2}{4},\quad 
    \avg\ls{\|\mc{E}_t^{(2)}\|^2}\leq 4d^2\tildefLipschitz^2, 
\end{align*}
\begin{equation}
    \begin{aligned}
    \avg\ls{\|\mc{E}_t^{(3)}\|^2}&\leq  \frac{d^2}{\delta_t^2}\fyLipschitz^2\bigg(2\alpha^t e_0 + 
    2C\sum_{k=0}^{t-1}\alpha^{t-k} \eta_k^2 \\&\quad + C\sum_{k=0}^{t-1}\alpha^{t-k}\delta_k^2\bigg),
    \end{aligned}
\end{equation}
where \(C\) is a scalar and \(e_0 = \|y_0^{(0)}-S(x_0)\|^2\). 
\end{lemma}

The stated bounds on \(\avg\ls{\|\mc{E}_t^{(1)}\|^2}\) and \(\avg\ls{\|\mc{E}_t^{(2)}\|^2}\) are inspired by the literature on two-point zeroth-order gradient estimators \cite{spall1997one,flaxman2004online}. We use the Lipschitz property of \(f(x,\cdot)\) to bound 
\begin{align*}
{\|\mc{E}^{(3)}_t\|^2}
    \leq  2\frac{d^2}{\delta_t^2}\fyLipschitz^2\Big(\underbrace{\|y_{t}^{(K)}-\br(\hat{x}_t)\|^2}_{\text{Term A}}+\underbrace{\|\tilde{y}_t^{(K)}-\br(x_t)\|^2}_{\text{Term B}}\Big). 
\end{align*}
Following which, Term A and Term B are recursively bounded. 

A detailed proof is provided in the Appendix.

\subsection{Non-convergence to saddle points}
In this section, we show that the updates in \eqref{eq: UpperLevel} does not converge to a saddle point. Towards that goal, we make the following assumption that posits that the function \(\tilde{f}(\cdot)\) satisfy the strict saddle property.

\begin{assumption}\label{assm: StrictSaddle}
For any saddle point \(x^\ast\) of \(\tilde{f}\), it holds that \(\lambda_{\min}(\nabla^2\tilde{f}(x^\ast)) < 0\).
\end{assumption}
In the following theorem, we formally state the non-convergence result.
\begin{theorem}\label{thm: NonconvergenceSaddle}
Let Assumption \ref{assm: BasicAssumptionSetup}-\ref{assm: StrictSaddle} hold. For \(\epsilon>0\) there exists a time \(T_{\epsilon}\) such that for any saddle point \(x^\ast\) of \(\tilde{f}\) it holds that \(
    \avg\ls{\|x_t-x_\ast \|^2} \geq \epsilon,  \forall \ t\geq T_{\epsilon}.
\)
\end{theorem}

To prove Theorem \ref{thm: NonconvergenceSaddle},  an asymptotic pseudo-trajectory of \eqref{eq: UpperLevel} is constructed. We then show that the asymptotic pseudo-trajectory almost surely avoids saddle point.

\paragraph{Proof of Theorem \ref{thm: NonconvergenceSaddle}}
The proof follows by contradiction.
Suppose there exists a saddle point \(x^\ast\) such that 
\begin{align*}
\lim_{t\ra\infty}    \avg\ls{\|x_{t}-x^\ast\|^2} = 0. 
\end{align*}
This implies that for any \(\epsilon>0\) there exists an integer \(T_{\epsilon}\) such that for all \(t\geq T_{\epsilon}\) it holds that 
\begin{align}\label{eq: Th2Pf1}
    \avg\ls{\|x_{t+s}-x^\ast\|^2} \leq \epsilon/4\quad \forall s\geq 0. 
\end{align}

Next, for any arbitrary point \(x_t\) along the trajectory \eqref{eq: UpperLevel}, we define a dynamics parametrized by \(\hat{x}_t=x_t+\delta_tv_t\), as follows  
\begin{align*}
     z_{s+1}(\hat{x}_t) := z_s(\hat{x}_t) - \eta_{t+s} \nabla \tilde{f}(z_s(\hat{x}_t)), 
\end{align*}
where \(z_0(\hat{x}_t) = \hat{x}_t\). From Lemma \ref{lem: PseudoTrajectory},  we know that for any \(\epsilon>0\) and positive integer \(S\) there exists \(\tilde{T}_{\epsilon,S}\) such that 
\begin{align}\label{eq: Th2Pf2}
    \sup_{s\in [0,S]}\avg\ls{\| z_{s}(\hat{x}_t) - x_{t+s}\|^2} \leq \epsilon/4 \quad \forall \ t\geq \tilde{T}_{\epsilon,S}. 
\end{align}
Next, note that 
\begin{align*}
    \|z_s(\hat{x}_t)-x^\ast\|^2 &\leq 2\|z_s(\hat{x}_t)-x_{t+s}\|^2 + 2\|x_{t+s}-x^\ast\|^2.  
\end{align*} 
Therefore, combining \eqref{eq: Th2Pf1}-\eqref{eq: Th2Pf2}, we observe that for every \(t\geq \max\{T_{\epsilon},T_{\epsilon,S}\}\)
\begin{align*}
    \|z_s(\hat{x}_t)-x^\ast\|^2 &\leq \epsilon, \quad \forall \ s\in [0,S]
\end{align*}

But from \cite{lee2016gradient} we know that  for gradient descent with random initialization almost surely avoids converging to saddle point \footnote{More specifically, we use the results from \cite[Proposition 8]{lee2016gradient}. Even though the results in \cite[Proposition 8]{lee2016gradient} hold for gradient descent update with constant step-size, we can use this result for decaying step size in our context as well. This is because the proof of \cite[Proposition 8]{lee2016gradient} only requires each step of the gradient update to be diffeomorphism, which holds in our setting as the step-sizes are always non-negative and decaying.} there exists \(S_{\epsilon}\) such that for all \(s\geq S_{\epsilon}\) it holds that 
\begin{align*}
    \|z_{s}(\hat{x}_t)-x^\ast\|^2 \geq 2\epsilon.
\end{align*}
 This establishes contradiction.

The following lemma prescribes the pseudo-trajectory. 
\begin{lemma}\label{lem: PseudoTrajectory} Let \(x_t\) be an arbitrary point along the trajectory \eqref{eq: UpperLevel}.  Define a dynamics parametrized by \(\hat{x}_t=x_t+\delta_tv_t\), such that  \(
    z_{s+1}(\hat{x}_t) := z_s(\hat{x}_t) - \eta_{t+s} \nabla \tilde{f}(z_s(\hat{x}_t)), 
\)
where \(z_0(\hat{x}_t) = \hat{x}_t\) and it holds that  for any positive integer \(L\), we have
\begin{align*}
    \lim_{t\rightarrow \infty} \sup_{s\in [0,L]} \avg\ls{\|x_{t+s}-z_s(\hat{x}_t)\|^2} = 0.
\end{align*}
\end{lemma}
A detailed proof of Lemma \ref{lem: PseudoTrajectory} is provided in the Appendix.

%% file: Section/algorithm.tex
\begin{algorithm}
    \caption{ Follower Agnostic Stackelberg Optimization Algorithm}
 \label{alg: ZerothOrderTwoPointAlgorithm}
\textbf{Input:} {Time horizon \(T\),  Initial conditions \(y_0^{(0)} \in Y, \tilde{y}_0^{(0)}\in Y, x_0\in X\), Step sizes \((\eta_t)\), 
 
 \quad\quad \ \  Perturbation radius \((\delta_t)\)}
 \For{\(t= 0,1,...T-1\)}
 {
 Sample \(v_t\sim \textsf{Unif}(\mathcal{S}(\R^d))\)
 
 Assign \(\hat{x}_t = x_t+\delta_t v_t\)
 
 \For{\(k=0,1,...,K-1\)}{  
 Update \(\tilde{y}_{t}^{(k+1)} = H_k(\tilde{y}_t^{(k)};x_t)\)
}

Update \(x_{t+1} = x_t - \stepX_t \frac{d}{\delta_t}\left( f(\hat{x}_t,y_t^{(K)}) - f(x_t,\tilde{y}_t^{(K)}) \right) v_t 
\) 

Set \(y_{t+1}^{(0)} =\tilde{y}_{t+1}^{(0)} = \tilde{y}_t^{(K)}\)
 }
\end{algorithm}

%% file: Section/numerics.tex
We numerically study the Algorithm \ref{alg: ZerothOrderTwoPointAlgorithm} in the context of incentive design in routing games (described in Section \ref{sec: MotivatingExample}). We consider the Sioux Falls transportation network, as depicted in Figure \ref{fig:SiouxFalls}(a). The latency function and network topology are inherited from \url{http://tinyurl.com/y4fm4nvt}. We consider a synthetic demand of \((1,2,3,2,2,1)\) units, respectively, between o-d pairs \(\mathcal{Z} = ((1,20), (13,2), (20,1), (10,13), (11,20), (4,21))\). 

\begin{figure}
    \centering
    \includegraphics[width=0.45\textwidth]{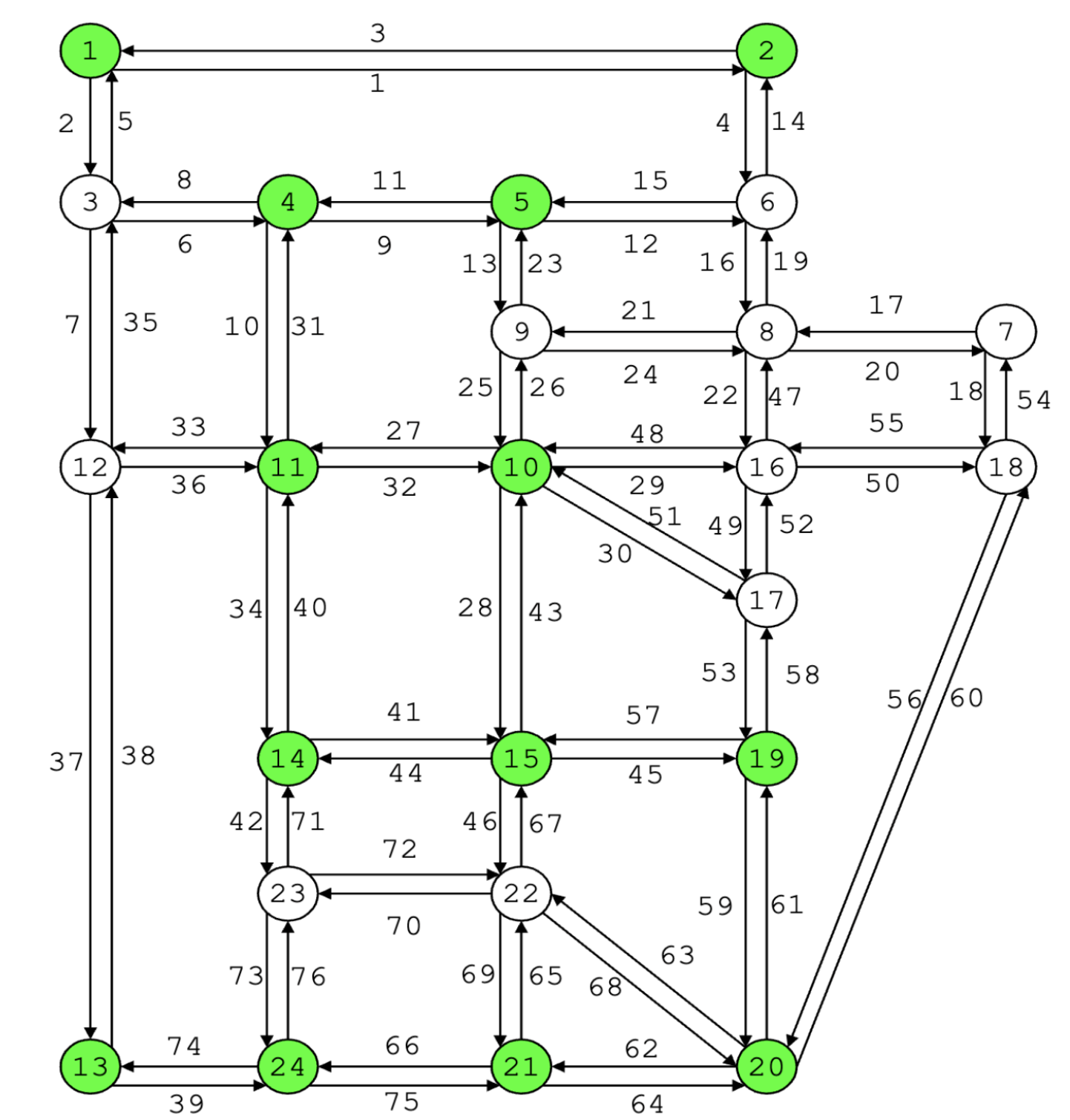}
    \caption{Schematic depiction of Sioux Falls transportation network network. The numbers on the edges and nodes are identifiers.}
    \label{fig:enter-label}
\end{figure}
\begin{figure}
    \centering
\includegraphics[width=0.45\textwidth]{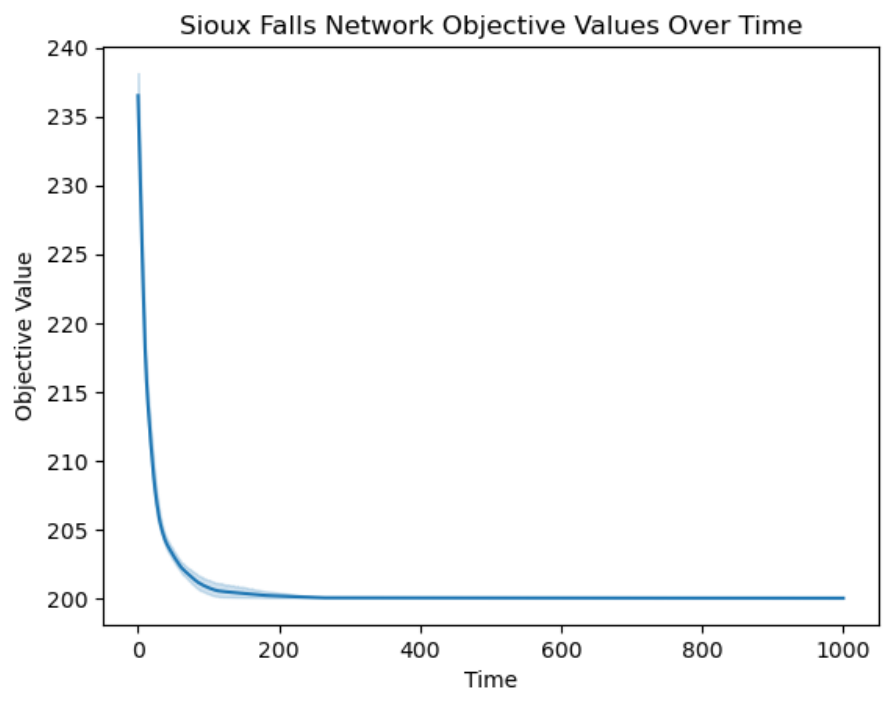}
    \caption{The evolution of planners objective function with iterates of the algorithm. The shaded blue region denotes the confidence interval calculated over 12 runs.}
    \label{fig:SiouxFalls}
\end{figure}
The incentive designer can set tolls on each edge of the network. In response, unknown to the planner, the travelers alter their route selection as per a gradient rule. More formally, given a toll \(x\in\mathbb{R}^{|\mc{E}|}\), we consider that the route choices made by the travelers are updated by descending along the gradient of the potential function \(\Phi(\cdot,x)\) (cf. \eqref{eq: WE}). Note that, for any \(x\in \mb{R}^{|\mc{E}|}, z\in \mathcal{Z}, r\in \mathcal{R}_z\), the gradient is 
\begin{align*}
    &\frac{\partial \Phi(y,x)}{\partial y_{rz}} = \sum_{e\in \mc{E}}\ell_e(w_e(y))\frac{\partial w_e(y)}{\partial y_{rz}} + x_e \frac{\partial w_e(y)}{\partial y_{rz}} \\ 
    &\underset{(i)}{=} \sum_{e\in \mc{E}}\ell_e(w_e(y))\mathbf{1}(e\in r) + x_e \mathbf{1}(e\in r) \underset{(ii)}{=}c_r(y,x)
\end{align*}
where \((i)\) is due to \eqref{eq: EdgeCongestion} and \((ii)\) follows from \eqref{eq: route_cost}. Consequently, the gradient update takes the following form: for every \(z\in \mathcal{Z}\): 
\begin{align}\label{eq: H_k_traffic}
    y_z^{(k+1)} = \mathcal{P}_{Y_z}\left((y_{rz}^{(k)} -  \gamma c_r(q^k, p))_{r\in \mc{R}_z}\right).
\end{align}

We simulate \(12\) runs of Algorithm \ref{alg: ZerothOrderTwoPointAlgorithm} with \(H_k\) given in \eqref{eq: H_k_traffic}. We set \(T = 1000\) and \(K = 3\). The initial route flow vector \(y_0^{(0)}\) and \(\tilde{y}_0^{(0)}\) are randomly initialized. We set initial tolls uniformly randomly between \([0,0.1]\). We set the step size \(\eta_t = 6(t+1)^{-1/2}\), \(\delta_t = 0.3 \cdot (t+1)^{-1/4}\), \(\gamma = 0.005\) and \(\lambda = 0.01\). 

In Figure \ref{fig:SiouxFalls}(b), we show the leader's objective value as function of time iterates \(t \in [T]\). 
We observe that all trajectories converge to same objective value even with random initializations. This shows that the convergent point is perhaps a global optimizer.

%% file: Section/conclusions.tex
We propose an efficient algorithm for Stackelberg games which converges to a stationary point at a rate of \(\mathcal{O}(T^{-1/2})\) and asymptotically reaches a local Stackelberg equilibrium. The algorithm is designed so that the leader does not need to know any information about the game structure at lower-level and updates its strategies by only querying for the followers response to its strategy. However, in this work we assume that follower's equilibrium strategy is singleton, Lipschitz and the leader's hyper-objective is differentiable. An interesting direction of future work is to relax this assumption. 

%% file: Appendix/UsefulResults.tex
\paragraph{Properties of zeroth-order gradient estimator.}
In this section, we state relevant properties of the zeroth-order gradient estimator used in the proposed algorithm.
Define the $\queryRadius$-smoothed loss function
$\tilde{f}_\queryRadius: \R^d \ra \R$ by $\tilde{f}_\queryRadius(x) := \E_{\overline{v} \sim \textsf{Unif}(B^d)}[\tilde{f}(x + \queryRadius \overline{v})]$,
where
$\mc{B}(\R^d)$ denotes the $d$-dimensional unit open ball in
$\R^d$. We further define $\queryRadius \cdot \mc{S}(\R^{d}) := \{\queryRadius v: v \in \mc{S}(\R^d) \}$ and $\queryRadius \cdot \mc{B}(\R^d) := \{\queryRadius \overline{v}: \overline{v} \in \mc{B}(\R^d) \}$. Finally, we use $\textsf{vol}_d(\cdot)$ to denote the volume of a set in $d$ dimensions.

\begin{lemma}[\cite{spall1997one,flaxman2004online,maheshwari2022zeroth}
]\label{Prop: App, ZO grad mean, variance}
Let $\tilde{F}(x;\delta,v) = \frac{d}{\queryRadius} \cdot \left(\tilde{f}(\hat{x})-\tilde{f}(x)\right) v$ where \(\hat{x}=x+\delta v\) and \(F(x)=\nabla \tilde{f}_{\delta}(x)\). Then the following holds: 
\begin{align} 
\label{Eqn: Prop, ZO grad mean, equality}
    \E_{v \sim \textsf{Unif}(\mc{S}(\R^{d}))} \big[\tilde{F}(x; \queryRadius, v) \big] &= \nabla \tilde{f}_\queryRadius(x). 
\end{align}
\end{lemma}
\textit{Proof.}
 Observe that  $\tilde{f}_\queryRadius(x) = \avg_{v \sim \textsf{Unif}(B^d)}[\tilde{f}(\hat{x})]$ 
 where \(\hat{x}=x+\delta v\) and
 $\tilde{F}(x; \queryRadius, v) = \frac{d}{\queryRadius} \cdot \left(\tilde{f}(x + \queryRadius v)-\tilde{f}(x)\right) v$ for each $x \in \R^d$, $\queryRadius > 0$, and $v \in \mc{S}(\R^d)$. We compute
\begin{align} \nonumber
    \nabla \tilde{f}_{\delta} (x) &= \nabla \E_{\overline{v} \sim \textsf{Unif}(\mc{B}(\R^d))} \big[ \tilde{f}(x + \queryRadius \overline{v}) \big] \\ \nonumber
    &= \nabla \E_{\overline{v} \sim \textsf{Unif}(\queryRadius \cdot \mc{B}(\R^d))} \big[ \tilde{f}(x + \overline{v}) \big] \\ \nonumber
    &= \frac{1}{\vol_d(\queryRadius \cdot \mc{B}(\R^d))} \cdot \nabla \left( \int_{\queryRadius \cdot B^d} \tilde{f}(x + \overline{v}) \hspace{0.5mm} d\overline{v} \right) \\ \label{Eqn: Stokes' Theorem, Application} 
    &= \frac{1}{\vol_d(\queryRadius \cdot \mc{B}(\R^d))} \cdot \int_{\queryRadius \cdot \mc{S}^{d-1}} \tilde{f}(x + v) \cdot \frac{v}{\Vert v \Vert_2} \hspace{0.5mm} dv, 
\end{align}
where \eqref{Eqn: Stokes' Theorem, Application} follows by Stokes' Theorem  \cite{Lee2006IntroductionToSmoothManifolds} which implies that 
\begin{align*}
    \nabla \int_{\queryRadius \cdot \mc{B}(\R^d)} \tilde{f}(x + \overline{v}) \hspace{0.5mm} d\overline{v} = \int_{\queryRadius \cdot \mc{S}(\R^d)} \tilde{f}(x + v) \cdot \frac{v}{\Vert v \Vert_2} \hspace{0.5mm} dv.
\end{align*}

Next note that 
\begin{align*}
     \E_{v \sim \textsf{Unif}(\mc{S}^{d-1})} \big[ \tilde{F}(x; \queryRadius, v) \big] &= \frac{d}{\queryRadius} \cdot \E_{v \sim \textsf{Unif}(\mc{S}(\R^{d}))} \ls{ \lr{\tilde{f}(x + \queryRadius v)-\tilde{f}(x)} v } \\ \nonumber
    &= \frac{d}{\queryRadius} \cdot \E_{v' \sim \textsf{Unif}(\queryRadius \cdot \mc{S}(\R^d))} \Bigg[ \tilde{f}(x + v') \cdot \frac{v'}{\Vert v' \Vert_2} \Bigg] \\ \nonumber
    &= \frac{d}{\queryRadius} \cdot \frac{1}{\vol_{d-1}(\queryRadius \cdot \mc{S}(\R^d))} \cdot \int_{\queryRadius \cdot \mc{S}(\R^d)} \tilde{f}(x+ v) \cdot \frac{v}{\Vert v \Vert_2} \hspace{0.5mm} dv,
\end{align*}

The equality \eqref{Eqn: Prop, ZO grad mean, equality} now follows by observing that
\begin{align*}
    \frac{\vol_{d-1}(\queryRadius \cdot \mc{S}(\R^d))}{\vol_{d}(\queryRadius \cdot \mc{B}(\R^d))} = \frac{d}{\delta}.
\end{align*}
That is, surface-area-to-volume ratio of sphere in \(\R^d\) or radius \(\delta\) is $d/\queryRadius$. 
This concludes the proof.

\paragraph{Discrete-Gronwall Inequality.}
\begin{lemma}[\cite{borkar2009stochastic}]\label{lem: DiscreteGronwall}
For a  non-negative sequence \((s_n)\)  such that 
\begin{align}\label{eq: GIDyn}
    s_{n+1}\leq {C}+L\lr{\sum_{m=0}^{n}s_m }, 
\end{align}
the following inequality holds: $s_{n+1} \leq \tilde{C}\exp(Ln)$ where \(\tilde{C} = \max\{s_0,C\}\).
\end{lemma}
\textit{Proof.}
The proof is taken from \cite{borkar2009stochastic}. We provide details here for the sake of completeness.

Define \(S_n = \sum_{m=0}^{n}s_m\). Then \eqref{eq: GIDyn} can be equivalently written as 
\begin{align*}
    S_{n+1}-S_n \leq \tilde{C}+LS_n
\end{align*}
where \(\tilde{C} = \max\{s_0,C\}\). 
Thus, it follows that 
\begin{align*}
    S_{n+1} &\leq  (1+L)S_n + \tilde{C}\\
    &\leq (1+L)^{n+1}S_0 + \sum_{k=0}^{n}\tilde{C}(1+L)^{n-k} \\ 
    &\leq \tilde{C}\sum_{k=0}^{n+1}(1+L)^{n+1-k} \\ 
    &\leq \tilde{C}\sum_{k=0}^{n+1}\exp(L(n+1-k)) \\ 
    &\leq \tilde{C}\int_{0}^{n+1}\exp(L(n+1-\tau))d\tau \\
     &= \frac{\tilde{C}}{L}\lr{\exp(L(n+1))-1}
\end{align*}
Substituting this in \eqref{eq: GIDyn} we obtain 
\begin{align*}
    x_{n+1}\leq \tilde{C} + LS_n \leq \tilde{C} + \tilde{C} \lr{\exp(L(n+1))-1} = \tilde{C}\exp(L(n+1))
\end{align*}
This concludes the proof.


%% file: Appendix/auxiliaryLemma.tex
\setcounter{lemma}{0}
\renewcommand{\thelemma}{\arabic{lemma}}

\begin{lemma}[Restated]\label{lem: TildeFPotentialAppendix}
If \(\bar{\eta}\leq d/(2\tildefSmooth)\) then 
\begin{align}
    \avg\ls{\tilde{f}(x_{t+1})} \leq \avg\ls{ \tilde{f}(x_t)}-\frac{\eta_t}{2}\avg\ls{\|\nabla \tilde{f}(x_t)\|^2 }\notag+ \eta_t \avg\ls{\|\mc{E}_t^{(1)}\|^2}+ \eta_t \avg\ls{\|\mc{E}_t^{(3)}\|^2}+{\tildefSmooth \eta_t^2} \avg\ls{\|\mc{E}_t^{(2)}\|^2} 
\end{align}
\end{lemma}
\textit{Proof.}
Note that by smoothness of \(\tilde{f}\) function
\begin{align}\label{eq: SmoothnessTildeF}
    \tilde{f}(x_{t+1}) &\leq \tilde{f}(x_t) + \lara{\nabla \tilde{f}(x_t),x_{t+1}-x_t} + \frac{\tildefSmooth}{2}\|x_{t+1}-x_t\|^2 \notag \\
    &=\tilde{f}(x_t) -\eta_t \lara{\nabla \tilde{f}(x_t),\hat{F}(x_t;\delta,v_t)} + \frac{\tildefSmooth\eta_t^2}{2}\|\hat{F}(x_t;\delta_t,v_t)\|^2 \notag \\ 
    &=\tilde{f}(x_t) -\eta_t \lara{\nabla \tilde{f}(x_t),\nabla \tilde{f}(x_t)+\mc{E}^{(1)}_t+\mc{E}^{(2)}_t+\mc{E}^{(3)}_t} + \frac{\tildefSmooth\eta_t^2}{2}\|\nabla \tilde{f}(x_t)+\mc{E}^{(1)}_t+\mc{E}^{(2)}_t+\mc{E}^{(3)}_t\|^2 
\end{align}
where we define 
\begin{align*}
    \mc{E}^{(1)}_t &:=\nabla \tilde{f}_{\delta}(x_t) -\nabla \tilde{f}(x_t)\\ 
     \mc{E}^{(2)}_t&:=
     \frac{d}{\delta_t}(\tilde{f}(\hat{x}_t)-\tilde{f}(x_t))v_t-\nabla \tilde{f}_{\delta_t}(x_t)
     \\
     \mc{E}^{(3)}_t&:=\frac{d}{\delta_t}\left((f(\hat{x}_t, y_{t}^{(K)})v_t -\tilde{f}(\hat{x}_t)v_t ) -(f({x}_t, \tilde{y}_{t}^{(K)})v_t -\tilde{f}({x}_t)v_t ) \right). 
\end{align*}

Taking expectation on both sides of \eqref{eq: SmoothnessTildeF} we obtain 
\begin{align*}
    \avg\ls{\tilde{f}(x_{t+1})} &\leq 
     \avg\ls{ \tilde{f}(x_t)}
     -\eta_t
     \avg\ls{ 
     \lara{
     \nabla \tilde{f}(x_t),\nabla\tilde{f}(x_t)+\mc{E}^{(1)}_t+\mc{E}^{(2)}_t+\mc{E}^{(3)}_t
     } 
     }
    \\&\quad +\frac{\tildefSmooth \eta_t^2}{2}\avg\ls{\|\nabla \tilde{f}(x_t)+\mc{E}^{(1)}_t+\mc{E}^{(2)}_t+\mc{E}^{(3)}_t\|^2} 
    \\ 
    &= \avg\ls{ \tilde{f}(x_t)}-\eta_t\avg\ls{ \lara{\nabla \tilde{f}(x_t),\nabla \tilde{f}(x_t)+\mc{E}^{(1)}_t+\mc{E}^{(3)}_t} }+\frac{\tildefSmooth \eta_t^2}{2}\avg\ls{\|\nabla \tilde{f}(x_t)+\mc{E}^{(1)}_t+\mc{E}^{(2)}_t+\mc{E}^{(3)}_t\|^2} \\ 
    &\leq \avg\ls{ \tilde{f}(x_t)}-\eta_t\avg\ls{ \lara{\nabla \tilde{f}(x_t),\nabla \tilde{f}(x_t)+\mc{E}^{(1)}_t+\mc{E}^{(3)}_t} }\\&\quad +{\tildefSmooth \eta_t^2}\lr{\avg\ls{\|\nabla \tilde{f}(x_t)+\mc{E}^{(1)}_t+\mc{E}^{(3)}_t\|^2} + \avg\ls{\|\mc{E}_t^{(2)}\|^2} } 
\end{align*}
where the first equality and last inequality follows by noting that \(\avg\ls{\mc{E}_t^{(2)}|x_t} = 0\) from Lemma \ref{alg: ZerothOrderTwoPointAlgorithm}. Next, choosing \(\eta_t\leq \frac{1}{2\tildefSmooth}\) we obtain 
\begin{align*}
    \avg\ls{\tilde{f}(x_{t+1})} &\leq  \avg\ls{ \tilde{f}(x_t)}-\frac{\eta_t}{2}\bigg(2\avg\ls{ \lara{\nabla \tilde{f}(x_t),\nabla \tilde{f}(x_t)+\mc{E}^{(1)}_t+\mc{E}^{(3)}_t} }-\avg\ls{\|\nabla \tilde{f}(x_t)+\mc{E}^{(1)}_t+\mc{E}^{(3)}_t\|^2} \bigg)\\&\quad +{\tildefSmooth \eta_t^2}\lr{ \avg\ls{\|\mc{E}_t^{(2)}\|^2} }  \\ 
    &= \avg\ls{ \tilde{f}(x_t)}-\frac{\eta_t}{2}\lr{\avg\ls{\|\nabla \tilde{f}(x_t)\|^2 } - \avg\ls{ \|\mc{E}_t^{(1)} + \mc{E}_t^{(3)}\|^2 } }+{\tildefSmooth \eta_t^2} \avg\ls{\|\mc{E}_t^{(2)}\|^2} 
\end{align*} 
where the equality follows by completing the squares.

\begin{lemma}[Restated]\label{lem: ErrorBoundsAppendix}
The errors \(\avg\ls{\|\mc{E}_t^{(i)}\|^2 }\) for \(i\in \{1,2,3\}\) are bounded as follows 
\begin{equation}
    \begin{aligned}
    &\avg\ls{\|\mc{E}_t^{(1)}\|^2}\leq \frac{\tildefSmooth^2\delta_t^2d^2}{4},\quad 
    \avg\ls{\|\mc{E}_t^{(2)}\|^2}\leq 4d^2\tildefLipschitz^2 \\
    &\avg\ls{\|\mc{E}_t^{(3)}\|^2}\leq  \frac{d^2}{\delta_t^2}\fyLipschitz^2\left(2\alpha^t e_0 + 
    2C_6\sum_{k=0}^{t-1}\alpha^{t-k} \eta_k^2 + C_4\sum_{k=0}^{t-1}\alpha^{t-k}\delta_k^2\right)
    \end{aligned}
\end{equation}
\end{lemma}
\input{Section/ProofOfErrorLemma}

\begin{lemma}[Restated]\label{lem: PseudoTrajectoryAppendix}
The trajectory \(z_s(\cdot)\) is an asymptotic pseudotrajectory of \eqref{eq: UpperLevel}. That is, for any positive integer \(S\)
\begin{align*}
    \lim_{t\rightarrow \infty} \sup_{s\in [0,S]} \avg\ls{\|x_{t+s}-z_s(\hat{x}_t)\|^2} = 0
\end{align*}
\end{lemma}
\textit{Proof.}
For any \(t\), we see that 
\begin{align*}
    &\|x_{t+s+1}-z_{s+1}(x_t)\|^2 = \|x_{t+s}-\eta_{t+s}\ZO(x_t,\delta_t,v_t) - z_s(\hat{x}_t) + \eta_{t+s}\nabla \tilde{f}(z_s(\hat{x}_t))\|^2 \\ 
    &\leq  2\|x_{t+s}-z_s(\hat{x}_t)\|^2 + 2\eta_{t+s}^2\|\ZO(x_{t+s},\delta_{t+s},v_{t+s})-\nabla \tilde{f}(z_s(\hat{x}_t))\|^2 \\ 
    &= 2\|x_{t+s}-z_s(\hat{x}_t)\|^2 + 2\eta_{t+s}^2 \|\ZO(x_{t+s},\delta_{t+s},v_{t+s})-\nabla \tilde{f}(x_{t+s})+\nabla \tilde{f}(x_{t+s})-\nabla \tilde{f}(z_s(\hat{x}_t))\|^2 \\ 
    &\leq 2(1+4\eta_{t+s}^2\tildefSmooth^2)\|x_{t+s}-z_s(\hat{x}_t)\|^2+4\eta_{t+s}^2\|\ZO(x_{t+s},\delta_{t+s},v_{t+s})-\nabla \tilde{f}(x_{t+s})\|^2 \\ 
    &= C_1\|x_{t+s}-z_s(\hat{x}_t)\|^2 + C_2 \eta_{t+s}^2 \|\ZO(x_{t+s},\delta_{t+s},v_{t+s})-\nabla \tilde{f}(x_{t+s})\|^2\\
    &\leq C_1^s\|x_{t}-z_0(\hat{x}_t)\|^2 + \sum_{k=1}^{s}C_2C_1^{k}\eta_{t+s-k}^2\|\ZO(x_{t+s-k},\delta_{t+s-k},v_{t+s-k})-\nabla \tilde{f}(x_{t+s-k})\|^2 
\end{align*}
where \(C_1 = 2(1+4(\bar{\eta}^2/d^2)\tildefSmooth^2), C_2 = 4\). 
Next we note that \(x_t=z_0(x_t)\). Consequently, 
\begin{align*}
   \avg\ls{ \|x_{t+s+1}-z_{s+1}(\hat{x}_t)\|^2} \leq \sum_{k=1}^{s}C_1^{s-k}\eta_{t+k}^2\lr{\avg\ls{\|\mc{E}^{(1)}_{t+k}\|^2 + \|\mc{E}^{(2)}_{t+k}\|^2+\|\mc{E}^{(3)}_{t+k}\|^2}} 
\end{align*}
Therefore, 
\begin{align*}
    &\sup_{s\in[0,S]} \avg\ls{ \|x_{t+s+1}-z_{s+1}(\hat{x}_t)\|^2} \\ 
   &\leq C_1^S\sum_{k=1}^{S}\eta_{t+k}^2\lr{D_1\delta_{t+k}^2+D_2+\frac{d^2}{\delta_{t+k}^2}\fyLipschitz^2\left(2\alpha^{t+k} e_0 + 
    2C_6\sum_{p=0}^{t+k-1}\alpha^{t+k-p } \eta_p^2 + C_4\sum_{p=0}^{t+k-1}\alpha^{t+k-p}\delta_p^2\right)} \\ 
    &\leq \mc{O}\lr{ C_1^S\lr{\eta_t^2\delta_t^2S+ \eta_t^2 S+\eta_t \alpha^{t} S + \underbrace{\eta_t\sum_{k=1}^{S} \frac{\eta_{t+k}}{\delta_{t+k}^2}\sum_{p=0}^{t+k-1}\alpha^{t+k-p}\eta_p^2}_{\text{Term A}} + \underbrace{\eta_t \sum_{k=1}^{S}\frac{\eta_{t+k}}{\delta_{t+k}^2}\sum_{p=0}^{t+k-1}\alpha^{t+k-p}\delta_p^2 }_{\text{Term B}}}}
\end{align*}
Next, we analyze Term A and Term B by substituting \(\eta_t = \bar{\eta} (t+1)^{-1/2}d^{-1}, \delta_t = \bar{\delta} (t+1)^{-1/4}d^{-1/2}\). First, we see that 
\begin{align*}
    \text{Term A} &= \eta_t\sum_{k=1}^{S} \frac{\eta_{t+k}}{\delta_{t+k}^2}\sum_{p=0}^{t+k-1}\alpha^{t+k-p}\eta_p^2 \\  
    &= \frac{\bar{\eta}}{\bar{\delta}^2}\eta_t\sum_{k=1}^{S} \sum_{p=0}^{t+k-1}\alpha^{t+k-p}\eta_p^2 \\
    &\leq \frac{\bar{\eta}}{\bar{\delta}^2}\eta_t\alpha^t\sum_{k=1}^{S} \alpha^{k} \sum_{p=0}^{t+k-1}\eta_p^2  \\ 
    &\leq \frac{\bar{\eta}^3}{d^2\bar{\delta}^2}\eta_t\alpha^t\sum_{k=1}^{S}\alpha^{k} \sqrt{t+k} \\
    &\leq \frac{\bar{\eta}^3}{d^2\bar{\delta}^2}\eta_{t}\frac{\alpha^{t+1}}{1-\alpha}\sqrt{t+S} 
\end{align*}

Next, we analyze Term B 
\begin{align*}
    \text{Term B} &= \eta_t\sum_{k=1}^{S} \frac{\eta_{t+k}}{\delta_{t+k}^2}\sum_{p=0}^{t+k-1}\alpha^{t+k-p}\delta_p^2 \\  
    &= \frac{\bar{\eta}}{\bar{\delta}^2}\eta_t\sum_{k=1}^{S} \sum_{p=0}^{t+k-1}\alpha^{t+k-p}\delta_p^2 \\
    &\leq \frac{\bar{\eta}}{\bar{\delta}^2}\eta_t\alpha^t\sum_{k=1}^{S} \alpha^{k} \sum_{p=0}^{t+k-1}\delta_p^2  \\ 
    &\leq \frac{\bar{\eta}}{d}\eta_t\alpha^t\sum_{k=1}^{S}\alpha^{k} \lr{t+k}^{3/4} \\
    &\leq \frac{\bar{\eta}^3}{d^2\bar{\delta}^2}\eta_{t}\frac{\alpha^{t+1}}{1-\alpha}\lr{t+S}^{3/4} 
\end{align*}

To summarize, we obtain

\begin{align*}
    &\sup_{s\in[0,S]} \avg\ls{ \|x_{t+s+1}-z_{s+1}(\hat{x}_t)\|^2} \\& 
    \leq \mc{O}\lr{ C_1^S\lr{\eta_t^2\delta_t^2S+ \eta_t^2 S+\eta_t \alpha^{t} S + \eta_{t}{\alpha^{t+1}}\sqrt{t+S}  +\eta_{t}{\alpha^{t+1}}\lr{t+S}^{3/4}}} \\ 
    &= \mc{O}\lr{ C_1^S\lr{\frac{S}{(t+1)^{3/2}}+ \frac{S}{(t+1)} +\frac{ \alpha^{t} S}{\sqrt{t+1}} + \frac{\sqrt{t+S}}{\sqrt{t+1}} {\alpha^{t+1}}  +\frac{\lr{t+S}^{3/4}}{\sqrt{t+1}}{\alpha^{t+1}}}} 
\end{align*}
Taking limit \(t\ra\infty\) we obtain the desired conclusion. 

\subsection{When does Assumption \ref{assm: FollowerUpdatesDiff} hold?}
Consider the scenario of bi-level optimization when the lower level problem is just a convex optimization problem with objective function \(g(x,\cdot)\) a popular choice of \(H\) is projected gradient descent: 
\begin{align}\label{eq: lowerlevelBilevel}
    y^{(k+1)}(x) = H(y^{(k)};x)= \proj_Y(y^{(k)}(x)-\stepY \nabla_y g(x,y^{(k)}(x)).
\end{align}

\begin{proposition}\label{lem: BoundYIter}
Consider the bilevel optimization problem with convex-lower level optimization problem \eqref{eq: lowerlevelBilevel}. Then for any \(x\in X\) the difference between \(y^{(K)}(x)\) and \(\tilde{y}^{(K)}(x)\) is bounded as
\begin{equation}
    \|y^{(K)}(\hat{x})-{y}^{(K)}(x)\| \leq K\gxLipschitz \gamma \delta_{t} \exp(\gamma \gyLipschitz K)
\end{equation}
\end{proposition}
\textit{Proof.}
We note that for any \(k\in [K]\)
\begin{align*}
    y^{(k)}(\hat{x}) &= \mc{P}_Y\left( y^{(k-1)}(\hat{x}) - \gamma\nabla g(\hat{x},y^{(k-1)}(\hat{x}))\right), \\ 
    {y}^{(k)}(x) &= \mc{P}_Y\left({y}^{(k-1)}(x) - \gamma \nabla g(x,{y}^{(k-1)}(x))\right), 
\end{align*}
where we impose that \(y^{(0)}(x) = y^{(0)}(\hat{x})\) due to Algorithm \ref{alg: ZerothOrderTwoPointAlgorithm}.
Then
\begin{align*}
    \|y^{(k)}(\hat x)-{y}^{(k)}(x)\|
    &= \|\proj_Y(y^{(k-1)}(\hat x)-\gamma \nabla g(\hat{x},y^{(k-1)}(\hat x)))-\proj_Y({y}^{(k-1)}(x)-\gamma \nabla g({x},{y}^{(k-1)}(x))\| \\ 
    &\leq \|y^{(k-1)}(\hat x)-{y}^{(k-1)}(x)\| + \gamma \|\nabla g(\hat{x},y^{(k-1)}(\hat x))-\nabla g({x},{y}_t^{(k-1)}(x))\|
    \\&\leq \gamma \sum_{\ell=0}^{k-1}\|\nabla g(\hat{x},y^{(\ell)}(\hat x))-\nabla g({x},{y}^{(\ell)}(x))\| \\ 
    &\leq \gamma\sum_{\ell=0}^{k-1}\left(\|\nabla g(\hat{x},y^{(\ell)}(\hat{x}))-\nabla g(\hat {x},y^{(\ell)}(x))\|+ \|\nabla g(\hat {x},y^{(\ell)}(x))-\nabla g({x},{y}^{(\ell)}(x))\|\right)\\
    &\leq \gamma  k\gxLipschitz\delta+\gamma\gyLipschitz\sum_{\ell=0}^{k-1}\|y^\ell(\hat{x}) - {y}^\ell(x)\| \\ &\leq K\gxLipschitz\gamma \delta + \gamma\gyLipschitz\sum_{\ell=0}^{k-1}\|y^\ell(\hat x) - {y}^\ell(x)\|
\end{align*}
By discrete Gronwall inequality stated in Lemma \ref{lem: DiscreteGronwall}  we obtain
\begin{align*}
    \|y^{(k)}(\hat x)-{y}^{(k)}(x)\|\leq K\gxLipschitz\gamma \delta \exp(\gamma \gyLipschitz k), \quad \forall \ k\in [K].
\end{align*}

%% file: Section/ProofOfErrorLemma.tex
\textit{Proof.}
We begin with bounding the terms \(\avg\ls{\|\mc{E}_t^{(i)}\|^2}\) for \(i\in \{1,2,3\}\). We note that 
\begin{align}\label{eq: Err1}
    &\avg\ls{\|\mc{E}_t^{(1)}\|^2} = \avg\ls{ \|\nabla \tilde{f}_{\delta_t}(x_t) -\nabla \tilde{f}(x_t)\|^2}\notag \\
    &=\avg\ls{\bigg\|\avg\ls{\frac{d}{\delta_t}(\tilde{f}(\hat{x}_t)-\tilde{f}(x_t))v_t-\nabla \tilde{f}(x_t)\bigg|x_t}\bigg\|^2}\notag \\
     & =\frac{d^2}{\delta_t^2}\avg\ls{\bigg\|\avg\ls{(\tilde{f}(\hat{x}_t)-\tilde{f}(x_t))v_t-\delta_t\avg[v_tv_t^\top]\nabla \tilde{f}(x_t)\bigg|x_t}\bigg\|^2}\notag \\
     &=\frac{d^2}{\delta_t^2}\avg\ls{\bigg\|\avg\ls{v_t\lr{\tilde{f}(\hat{x}_t)-\tilde{f}(x_t)-\delta_tv_t^\top\nabla \tilde{f}(x_t)}\bigg|x_t}\bigg\|^2}\notag \\ 
     &\leq  \frac{d^2}{\delta_t^2} \avg\ls{\|\tilde{f}(\hat{x}_t)-\tilde{f}(x_t)-\delta_tv_t^\top\nabla \tilde{f}(x_t)\|^2}
     \notag \\
     &\leq  \frac{d^2}{\delta_t^2}\frac{\tildefSmooth^2 \delta_t^4}{4} = \frac{\tildefSmooth^2\delta_t^2d^2}{4},
\end{align}
where the first inequality is due to Jensen's inequality and the last inequality is due to \(\tildefSmooth-\)smoothness of \(\tilde{f}\).
Next, we bound \(\|\mc{E}^{(2)}_t\|\) as follows 
\begin{align}\label{eq: E2_Error2}
&\|\mc{E}^{(2)}_t\|
    =\bigg\|\frac{d}{\delta_t}(\tilde{f}(\hat{x}_t)-\tilde{f}(x_t))v_t-\nabla \tilde{f}_{\delta_t}(x_t) \bigg\| \\& 
    \leq \bigg\|\frac{d}{\delta_t}(\tilde{f}(\hat{x}_t)-\tilde{f}(x_t))v_t\bigg\| + \bigg\|\mb{E}\left[\frac{d}{\delta_t}(\tilde{f}(\hat{x}_t)-\tilde{f}(x_t))v_t\Big|x_t\right]\bigg\| \notag  \\ 
    &\leq 2 \bigg\|\frac{d}{\delta_t}(\tilde{f}(\hat{x}_t)-\tilde{f}(x_t))v_t\bigg\| \notag \leq 2\frac{d}{\delta_t} \tildefLipschitz\|\hat{x}_t-x_t\| \leq 2{d} \tildefLipschitz 
\end{align}
Finally we bound \(\|\mc{E}^{(3)}_t\|\). Note that 
\begin{align}\label{eq: E3_ErrorLat}
    &{\|\mc{E}^{(3)}_t\|^2}
    = \bigg\| \frac{d}{\delta_t}\bigg((f(\hat{x}_t, y_{t}^{(K)})v_t -\tilde{f}(\hat{x}_t)v_t )\notag \\&\hspace{1cm}-(\tilde{f}({x}_t, \tilde{y}_{t}^{(K)})v_t -\tilde{f}({x}_t)v_t ) \bigg) \bigg\|^2 \notag \\
    &\leq 2\frac{d^2}{\delta_t^2}\bigg(\bigg\|f(\hat{x}_t, y_{t}^{(K)}) -\tilde{f}(\hat{x}_t,\br(\hat{x}_t))\bigg\|^2 \\&\hspace{1cm}+ \bigg\|f({x}_t, \tilde{y}_{t}^{(K)}) -\tilde{f}({x}_t,\br(x_t))\bigg\|^2\bigg)\notag 
    \\& \leq 2\frac{d^2}{\delta_t^2}\fyLipschitz^2\left(\underbrace{\|y_{t}^{(K)}-\br(\hat{x}_t)\|^2}_{\text{Term A}}+\underbrace{\|\tilde{y}_t^{(K)}-\br(x_t)\|^2}_{\text{Term B}}\right)
\end{align}
Recall from Assumption \ref{assm: FollowerUpdatesConvergence}   it holds that
\begin{align}\label{eq: TermA}
   &\text{Term A} = \|y_{t}^{(K)}-\br(\hat{x}_t)\|^2 \leq \alpha  \|y_{t}^{(0)}-\br(\hat{x}_t)\|^2 \notag \\
   &=\alpha  \|\tilde{y}_{t-1}^{(K)}-\br(\hat{x}_t)\|^2 =\alpha  \|\tilde{y}_{t-1}^{(K)}-\br({x}_{t-1})+\br({x}_{t-1})-\br(\hat{x}_t)\|^2\notag \\
    &=2 \alpha  (\|\tilde{y}_{t-1}^{(K)}-\br({x}_{t-1})\|^2+\brLipschitz\|\hat{x}_{t}-{x}_{t-1}\|^2)\notag \\
     &\leq  2 \alpha  (\|\tilde{y}_{t-1}^{(K)}-\br({x}_{t-1})\|^2+2\brLipschitz\|{x}_{t}-{x}_{t-1}\|^2+ 2\brLipschitz \|\delta_t v_t \|^2)\notag  \\ 
    &\leq  2 \alpha  (\|\tilde{y}_{t-1}^{(K)}-\br({x}_{t-1})\|^2+ 2\brLipschitz \|\delta_t v_t \|^2\notag \\&\hspace{1cm}+2\brLipschitz\eta_{t-1}^2\frac{d^2}{\delta_{t-1}^2}{\|f(\hat{x}_{t-1},y_{t-1}^{(K)})-f(x_{t-1},\tilde{y}_{t-1}^{(K)})\|^2}\notag )\notag  \\ 
    &\leq  2 \alpha  (\|\tilde{y}_{t-1}^{(K)}-\br({x}_{t-1})\|^2+ 2\brLipschitz \delta_{t-1}^2\notag\\&\hspace{1cm}+2\brLipschitz\eta_{t-1}^2\frac{d^2}{\delta_{t-1}^2}\underbrace{\|f(\hat{x}_{t-1},y_{t-1}^{(K)})-f(x_{t-1},\tilde{y}_{t-1}^{(K)})\|^2}_{\text{Term C}}) 
\end{align}
where the last inequality follows by noting that \((\delta_t)\) is a decreasing sequence. 

Similarly, we note that 
\begin{align}\label{eq: tildeyDiff}
    &\text{Term B} = \|\tilde{y}_{t}^{(K)}-\br({x}_t)\|^2 \leq \alpha  \|\tilde{y}_{t-1}^{(K)}-\br({x}_t)\|^2 \notag \\ 
    &=\alpha  \|\tilde{y}_{t-1}^{(K)}-\br({x}_{t-1})+\br({x}_{t-1})-\br({x}_t)\|^2 \notag \\
    &=2\alpha  (\|\tilde{y}_{t-1}^{(K)}-\br({x}_{t-1})\|^2+\brLipschitz\|{x}_{t}-{x}_{t-1}\|^2)\notag \\
    &\leq  2 \alpha  \bigg(\|\tilde{y}_{t-1}^{(K)}-\br({x}_{t-1})\|^2 \notag \\ &\hspace{1cm}+\brLipschitz\eta_{t-1}^2\frac{d^2}{\delta_{t-1}^2}\underbrace{\|f(\hat{x}_{t-1},y_{t-1}^{(K)})-f(x_{t-1},\tilde{y}_{t-1}^{(K)})\|^2}_{\text{Term C}}\bigg)
\end{align}
To finish the bounds for Term A and Term B, we need to bound Term C 
\begin{align}\label{eq: f_DiffLat}
   & \text{Term C} = \|f(\hat{x}_{t-1},y_{t-1}^{(K)})-f(x_{t-1},\tilde{y}_{t-1}^{(K)})\|^2 \notag \\ &= \|f(\hat{x}_{t-1},y_{t-1}^{(K)})-f(x_{t-1},y_{t-1}^{(K)})\notag \\&\hspace{1cm}+f(x_{t-1},y_{t-1}^{(K)})-f(x_{t-1},\tilde{y}_{t-1}^{(K)})\|^2 \notag \\ 
    &\leq 2(\|f(\hat{x}_{t-1},y_{t-1}^{(K)})-f(x_{t-1},y_{t-1}^{(K)})\|^2 \notag \\ &\hspace{1cm}+\|f(x_{t-1},y_{t-1}^{(K)})-f(x_{t-1},\tilde{y}_{t-1}^{(K)})\|^2) \notag \\ 
    &=2(\fxLipschitz^2\|\hat{x}_{t-1}-x_{t-1}\|^2+\fyLipschitz^2\|y_{t-1}^{(K)}-\tilde{y}_{t-1}^{(K)}\|^2) \notag  \\ 
    &=2(\fxLipschitz^2\delta_{t-1}^2+\fyLipschitz^2\underbrace{\|y_{t-1}^{(K)}-\tilde{y}_{t-1}^{(K)}\|^2}_{\text{Term D}})
\end{align}
Note that from Assumption \ref{assm: FollowerUpdatesDiff} it follows that \(\text{Term D} \leq C^2\delta_{t-1}^2\). 
Consequently, we obtain the following bound
\begin{align}\label{eq: F_diff}
    \text{Term C} &= \|f(\hat{x}_{t-1},y_{t-1}^{(K)})-f(x_{t-1},\tilde{y}_{t-1}^{(K)})\|^2 \notag \\&\leq  2(\fxLipschitz^2
    \delta_{t-1}^2+C_1^2\delta_{t-1}^2),
\end{align}
where \(C_1 := \fyLipschitz^2 C^2\). Consequently, we can bound \eqref{eq: TermA} as  
\begin{align*}
    &\|y_{t}^{(K)}-\br(\hat{x}_t)\|^2 \leq  2 \alpha  \bigg (\|\tilde{y}_{t-1}^{(K)}-\br({x}_{t-1})\|^2
   \\&\hspace{1cm} +\brLipschitz\eta_{t-1}^2\frac{d^2}{\delta_{t-1}^2}\left( 2(\fxLipschitz^2
    \delta_{t-1}^2+C_1\delta_{t-1}^2) \right)+ 2\brLipschitz \delta_{t-1}^2\bigg)\\
    &= 2 \alpha  \left(\|\tilde{y}_{t-1}^{(K)}-\br({x}_{t-1})\|^2+2 \brLipschitz \eta_{t-1}^2d^2 (\fxLipschitz^2
    +C_1) + 2\brLipschitz \delta_{t-1}^2\right) 
\end{align*}

Define \(e_t = \|y_t^{(K)}-\br(\hat{x}_{t})\|^2\), \(\tilde{e}_t = \|\tilde{y}_{t}^{(K)}-\br(\hat{x}_{t})\|^2\). Then we have 
\begin{align}\label{eq: ErrBound1}
    e_t &\leq \bar{\alpha} \left( \tilde{e}_{t-1} + C_2d^2\eta_{t-1}^2 + C_3d^2{\eta_{t-1}^2} + C_4\delta_{t-1}^2 \right), \\ 
    \tilde{e}_t &\leq \bar{\alpha} \left( \tilde{e}_{t-1} + C_2d^2\eta_{t-1}^2 + C_3d^2{\eta_{t-1}^2}\right),
\end{align}
where \(\bar{\alpha} = 2\alpha, C_2 = 2\brLipschitz\fxLipschitz^2, C_3=2\brLipschitz C_1,\) \(C_4 = 4\brLipschitz\) and \(C_5=2\brLipschitz\).
Consequently, it holds that 
\begin{align}\label{eq: TildeErrEqLat}
    \tilde{e}_{t} &\leq \bar\alpha\left(\tilde{e}_{t-1}+ C_2d^2\eta_{t-1}^2+ C_3 d^2 {\eta_{t-1}^2} \right)\notag \\
    &\leq \bar\alpha^t \tilde{e}_0+\sum_{k=0}^{t-1} \bar\alpha^{t-k} \left(C_2d^2\eta_k^2 + C_3 d^2 {\eta_k^2}\right) \notag \\ 
    &\leq \bar\alpha^t \tilde{e}_0 + C_6d^2 \sum_{k=0}^{t-1}\bar\alpha^{t-k} \eta_k^2,
\end{align}
where \(C_6 = C_2+C_3\). Moreover, we also note that  
\begin{align}\label{eq: ErrEqLat}
    e_t &\leq \bar\alpha (\tilde{e}_{t-1}+C_2d^2\eta_{t-1}^2 + C_3 d^2 {\eta_{t-1}^2}+C_4\delta_{t-1}^2) 
  \notag   \\ 
  &\leq \bar{\alpha}(\bar\alpha^{t-1} \tilde{e}_0 + C_6d^2 \sum_{k=0}^{t-2}\bar\alpha^{t-1-k} \eta_k^2 + C_6 d^2 \eta_{t-1}^2 + C_4\delta_{t-1}^2)
  \notag   \\
  &\leq \bar\alpha^{t} \tilde{e}_0 + C_6d^2 \sum_{k=0}^{t-1}\bar\alpha^{t-k} \eta_k^2  + \bar{\alpha}C_4\delta_{t-1}^2
\end{align}

Combining \eqref{eq: ErrEqLat} and \eqref{eq: TildeErrEqLat} in \eqref{eq: E3_ErrorLat}. We obtain that 
\begin{align}\label{eq: E_3FinalBound}
    \|\mc{E}^{(3)}_t\|^2 \leq \frac{d^2}{\delta_t^2}\fyLipschitz^2\left(2\bar\alpha^t e_0 + 
    2C_6\sum_{k=0}^{t-1}\bar\alpha^{t-k} \eta_k^2 + C_4\sum_{k=0}^{t-1}\bar\alpha^{t-k}\delta_k^2\right).
\end{align}